\documentclass[reqno,11pt]{amsart}

\usepackage{bbm}
\usepackage{amssymb,mathrsfs,amsfonts,paralist}
\usepackage[colorlinks=false,pdfborder={0 0 0}]{hyperref}
\usepackage[latin1]{inputenc}

\newtheorem{theorem}{Theorem}[section]
\newtheorem{coro}[theorem]{Corollary}
\newtheorem{lemma}[theorem]{Lemma}
\newtheorem{fact}[theorem]{Fact}

\theoremstyle{definition}
\newtheorem{definition}[theorem]{Definition}
\newtheorem{notation}[theorem]{Notation}
\newtheorem{remark}[theorem]{Remark}

\newcommand\cf{\mathbbm{1}}
\newcommand\C{\mathbb{C}}
\newcommand\E{\mathcal{E}}
\newcommand\N{\mathbb{N}}
\renewcommand\P{\mathbb{P}}
\newcommand\R{\mathbb{R}}
\newcommand\Z{\mathbf{Z}}
\renewcommand\a{\mathfrak{a}}
\renewcommand\b{\mathfrak{b}}
\renewcommand\c{\mathfrak{c}}
\newcommand\eps{\varepsilon}
\newcommand\vrh{\varrho}
\newcommand\Om{\Omega}
\newcommand\om{\omega}
\newcommand\wt{\widetilde}
\newcommand\la{\lambda}
\newcommand\al{\alpha}
\newcommand\ga{\gamma}
\newcommand\Ga{\Gamma}
\newcommand\si{\sigma}
\newcommand\Si{\Sigma}
\newcommand\del{\delta}
\newcommand\leqC{\lesssim}

\newcommand\dom{\mathcal{D}}
\newcommand\ov{\overline}
\newcommand{\supp}{\operatorname{supp}}
\renewcommand{\Re}{\operatorname{Re}}

\newcommand{\rot}{\operatorname{rot}}
\renewcommand{\div}{\operatorname{div}}
\newcommand{\grad}{\nabla}

\begin{document}

\title{$L^p$-spectral multipliers for some elliptic systems}

\author{Peer Christian Kunstmann}
\address{Department of Mathematics, Karlsruhe Institute of Technology (KIT),
Kaiserstr.\ 89, 76128 Karlsruhe, Germany}
\email{\href{mailto:peer.kunstmann@kit.edu}{peer.kunstmann@kit.edu}, \href{mailto:matthias.uhl@kit.edu}{matthias.uhl@kit.edu}}
\author{Matthias Uhl}


\begin{abstract}
We show results on $L^p$-spectral multipliers for Maxwell operators
with bounded measurable coefficients. We also present similar results
for the Stokes operator with Hodge boundary conditions and the Lam\'e 
system. Here we rely on resolvent estimates established recently by 
M.\ Mitrea and S.\ Monniaux.
\end{abstract}

\maketitle

\numberwithin{equation}{section}

\section{Introduction}

For self-adjoint operators $A\ge0$ in a Hilbert space $H$, the
spectral theorem establishes a functional calculus for bounded Borel
measurable functions $F\colon[0,\infty)\to\C$. This property is crucial
in countless applications in mathematical physics. In particular in the
context of non-linear phenomena one studies differential
operators and associated semigroup or resolvent operators also in
spaces $L^p$ for $p\neq 2$. In this context, the holomorphic
$H^\infty$-functional calculus, i.e.\ a functional calculus for bounded
holomorphic functions on a complex sector symmetric to the real half
line, has turned out to be a very useful tool. But if the operator is
self-adjoint in $L^2$ it might have a better functional calculus in
$L^p$ for $p\neq2$ for appropriate functions $F\colon[0,\infty)\to\C$. 

The classical result in this field is H\"ormander's spectral
multiplier theorem for $A=-\Delta$ on $\R^D$ (1960), cf.\ Theorem
\ref{SpecMultHor} below.
Various generalizations of this result have been given since then, in
several directions. Quite recently, considerable progress has been
made (\cite{COSY,DY,Dzi2,KriPre,KunUhl}) concerning operators for 
which the associated
semigroups satisfy generalized Gaussian bounds or Davies-Gaffney
estimates (cf.\ Section 2 for more details). In this paper we show that
these results can be applied to several elliptic systems, namely the
Maxwell operator, the Stokes operator with Hodge boundary conditions,
and the Lam\'e system. 

The Maxwell operator is of great importance in the studies of
electrodynamics. Following the outline in \cite[Chapter
6]{ColtonKress}, we briefly explain how an interest in its spectral
properties arises. The Maxwell equations
\begin{align*}
\rot\mathcal E+\partial_t\mathcal H=0\,,\qquad
\rot\mathcal H-\eps(\cdot)\partial_t\mathcal E=0\,,\qquad
\div\mathcal H=0\qquad\mbox{in }\Om
\end{align*}
govern the propagation of electromagnetic waves in a region
$\Om\subset\R^3$. Here, $\mathcal E\colon\Om\times\R\to\R^3$ and
$\mathcal H\colon\Om\times\R\to\R^3$ denote the electric and
magnetic field, respectively, whereas the matrix-valued function
$\eps(\cdot)\colon\Om\to\R^{3\times3}$ describes the electric
permittivity. The magnetic permeability was taken to be the identity
matrix and the electric conductivity to be zero. We take perfect
conductor boundary conditions
$$
\nu\times\mathcal E=0\,,\qquad
\nu\cdot\mathcal H=0\qquad \mbox{on }\partial\Om\,.
$$
If the waves behave time periodically with respect to the same
frequency $\om>0$, the ansatz $\mathcal E(x,t)=e^{-i\om t}E(x)$ and
$\mathcal H(x,t)=e^{-i\om t}H(x)$ leads to the {\em
time-harmonic Maxwell equations} $\rot E-i\om H=0$ and $\rot
H+i\om\eps(\cdot) E=0$. Elimination of $E$ finally yields
\begin{align*}
\rot\eps(\cdot)^{-1}\rot H-\om^2H&=0\qquad\mbox{in }\Om\,,\\
\div H&=0\qquad\mbox{in }\Om\,,\\
\nu\cdot H&=0\qquad\mbox{on }\partial\Om\,,\\
\nu\times\eps(\cdot)^{-1}\rot H&=0\qquad\mbox{on }\partial\Om\,.
\end{align*}
The operator $\rot\eps(\cdot)^{-1}\rot$ is what we call the 
\emph{Maxwell operator} and we shall study it in the following
setting. Let $\Om$ be a bounded
Lipschitz domain in $\R^3$ and $\eps(\cdot)\in
L^\infty(\Om,\C^{3\times3})$ be a matrix-valued function such that
$\eps(\cdot)^{-1}\in L^\infty(\Om,\C^{3\times3})$ and
$\eps(x)\in\C^{3\times3}$ is a positive definite, hermitian matrix
for almost all $x\in\Om$. We emphasize that no additional regularity
assumptions on $\eps(\cdot)$ are made. This is of interest in solid
state physics, e.g.\ for photonic crystals. Inspired by the approach in
\cite{MiMo}, we consider in $L^2(\Om,\C^3)$ the operator $A_2$
which is associated with the densely defined, sesquilinear form
\begin{align*}
\a(u,v):=\int_\Om\eps(\cdot)^{-1}\rot u\cdot\ov{\rot v}\,dx+
\int_\Om\div u~\ov{\div v}\,dx\qquad(u,v\in\dom(\a)),
\end{align*}
where $\dom(\a):=\{u\in L^2(\Om,\C^3)\,:\,\div u\in L^2(\Om,\C),
\,\rot u\in L^2(\Om,\C^3),\,\nu\cdot u|_{\partial\Om}=0\}$. Here and
in the following, $\nu(x)$ denotes the outer normal at a point $x$
of the boundary $\partial\Om$ and the operators $\div$ and $\rot$
are defined in the distributional sense (cf., e.g.\ \cite{Amrouche}).

The main task in order to apply the recent results on spectral multipliers
mentioned above is to establish generalized Gaussian estimates for the
semigroup $(e^{-tA_2})_{t>0}$ associated with the operator $A_2$ (cf.\ Theorem
\ref{GGE-Hodge}). We do this via Davies' perturbation method, 
and we thus obtain that a spectral multiplier theorem holds for
$A_2$ (cf.\ Theorem \ref{SpecMultHodgeLaplace}). We define the
Maxwell operator $M_2$ as the restriction of $A_2$ to the space of
divergence-free vector fields. Since the Helmholtz projection and $A_2$
are commuting (cf.\ Lemma \ref{HHPro-komm}), many properties of $A_2$
can be transferred to the Maxwell operator $M_2$. This includes in
particular the validity of the spectral multiplier theorem (cf.\
Theorem~\ref{SpecMultMaxwell}). 

Besides the Maxwell operator we study the Stokes operator with Hodge
boundary conditions in bounded Lipschitz domains via results from
\cite{MiMo}. Actually, this operator corresponds to the special case of
$\eps(x)$ being the identity matrix for every $x\in\Om$. Then the operator
$A_2$ equals the \emph{Hodge-Laplacian} (observe that \cite{MiMo} also
studied a Maxwell operator, but that this is different from ours).
M.\ Mitrea and S.\ Monniaux (\cite{MiMo}) proved that $A_2$ is then given by
\begin{align*}
\dom(A_2)&=\bigl\{u\in V(\Om)\,:\,\rot\rot u\in L^2(\Om,\C^3),\, \div u\in
H^1(\Om,\C),\, \nu\times\rot u|_{\partial\Om}=0\bigr\},
\\
A_2u&=\rot\rot u-\grad\div u=-\Delta u\qquad\mbox{for }
u\in\dom(A_2)
\end{align*}
and that $-A_2$ generates an analytic semigroup on $L^p(\Om,\C^3)$
for all $p\in(p_\Om,p_\Om')$, where $p_\Om'>3$ and
$1/p_\Om+1/p_\Om'=1$. As a consequence, they obtained that (minus)
the Stokes operator with Hodge boundary conditions, which is defined
as the restriction of the Hodge-Laplacian on the space of
divergence-free vector fields, also generates an analytic semigroup on
$L^p(\Om,\C^3)$ for all $p\in(p_\Om,p_\Om')$.
We show that even a spectral multiplier theorem holds for the Stokes
operator with Hodge boundary conditions (cf.\ Theorem
\ref{SpecMultStokes}). Our arguments rely on the proof of M.\ Mitrea
and S.\ Monniaux in which certain two-ball estimates for the
resolvents of the Hodge-Laplacian were verified. We shall prove that
these kinds of bounds entail generalized Gaussian estimates for the
corresponding semigroup operators (cf.\ Lemma \ref{GGE-SG-Res}) and
thus the same reasoning as for the Maxwell operator is possible for
getting Theorem \ref{SpecMultStokes}.

Finally, by using a similar approach based on \cite{MiMoLame}, we
verify generalized Gaussian estimates for the time-dependent Lam\'e
system equipped with homogeneous Dirichlet boundary conditions. Thus
we obtain a spectral multiplier theorem for the Lam\'e system (cf.\
Theorem \ref{SpecMultLame}).

Let us mention that the generalized Gaussian estimates we establish
for the elliptic systems in this paper have other consequences that
have not been mentioned in the literature so far. Application of a
result from \cite{BK2} yields boundedness of $H^\infty$-functional calculus
in the stated range of $L^p$-spaces. Of course, this weaker assertion
follows also from the results on spectral multipliers of the present
paper. Due to \cite[Corollary 1.5]{BKLeg} (one could also use results
due to W. Arendt or E.B. Davies), the spectrum of these operators in
$L^p$ does not depend on $p$ for the stated range of $L^p$-spaces.
Finally we note that, in general, pointwise Gaussian kernel estimates
for all the above operators fail.

\smallskip
Throughout this article, we make use of the following notation. For
$p\in[1,\infty]$ the conjugate exponent $p'$ is defined by
$1/p+1/p'=1$ with the usual convention $1/\infty:=0$. In the proofs,
the letters $b,C$ denote generic positive constants that are
independent of the relevant parameters involved in the estimates and
may take different values at different occurrences. We will often
use the notation $a\leqC b$ if there exists a constant $C>0$ such
that $a\leq Cb$ for two non-negative expressions $a,b$; $a\cong b$
stands for the validity of $a\leqC b$ and $b\leqC a$. Moreover, the
notation $|E|$ for a Lebesgue measurable subset $E$ of $\R^D$ stands
for the $D$-dimensional Lebesgue measure of $E$.

\section{Spectral multiplier theorems}

In this section we quote and discuss results on spectral multipliers. 
Let $(X,d,\mu)$ be a {\em space of homogeneous type} in
the sense of Coifman and Weiss, i.e.\ $(X,d)$ is a non-empty metric
space endowed with a $\sigma$-finite regular Borel measure $\mu$
with $\mu(X)>0$ which satisfies the so-called {\em doubling
condition}, that is, there exists a constant $C>0$ such that for all
$x\in X$ and all $r>0$
\begin{align}\label{doubling}
\mu(B(x,2r))\leq C\,\mu(B(x,r))\,,
\end{align}
where $B(x,r):=\{y\in X:\,d(y,x)<r\}$. It is easy to see that the
doubling condition \eqref{doubling} entails the {\em strong
homogeneity property}, i.e.\ the existence of constants $C,D>0$ such
that for all $x\in X$, all $r>0$, and all $\la\geq1$
\begin{align}\label{doublingDim}
\mu(B(x,\la r))\leq C\la^D\mu(B(x,r))\,.
\end{align}
In the sequel the value $D$ always refers to the constant in
(\ref{doublingDim}) which will be also called {\em dimension} of
$(X,d,\mu)$. Of course, $D$ is not uniquely determined.

There is a multitude of examples of spaces of homogeneous type. The
simplest one is the Euclidean space $\R^D$, $D\in\N$, equipped with
the Euclidean metric and Lebesgue measure. Bounded open subsets
of $\R^D$ with Lipschitz boundary endowed with the Euclidean metric
and Lebesgue measure form also spaces of homogeneous type (with
$\mu(B(x,r))\cong r^D$). More general definitions of spaces of 
homogeneous type can be found in \cite[Chapitre III.1]{CW} or in 
\cite[Section I.1.2]{Steinbook}.

\smallskip
Let $A$ be a non-negative, self-adjoint operator on the Hilbert
space $L^2(X)$. If $E_A$ denotes the resolution of the identity
associated with $A$, the spectral theorem asserts that the operator
\begin{align*}
F(A):=\int_0^\infty F(\la)\,dE_A(\la)
\end{align*}
is well defined and acts as a bounded linear operator on $L^2(X)$
whenever $F\colon[0,\infty)\to\C$ is a bounded Borel function.
Spectral multiplier theorems provide regularity assumptions on $F$
which ensure that the operator $F(A)$ extends from $L^p(X)\cap
L^2(X)$ to a bounded linear operator on $L^p(X)$ for all $p$ ranging
in some interval $I\subset (1,\infty)$ containing $2$. 

In 1960, L.\ H{\"o}rmander addressed this question for the Laplacian
$A=-\Delta$ on $\R^D$ during his studies on the boundedness of
Fourier multipliers on $\R^D$. In order to formulate his famous
result, we fix once and for all a non-negative cut-off function
$\om\in C_c^\infty(0,\infty)$ such that
\begin{align*}
\supp\om\subset(1/4,1) \qquad\mbox{and}\qquad
\sum_{n\in\Z}\om(2^{-n}\la)=1\quad\mbox{for all $\la>0$}\,.
\end{align*}
\begin{theorem}\label{SpecMultHor}\cite[Theorem 2.5]{Hor60}
If $F\colon[0,\infty)\to\C$ is a bounded Borel function such that
$$
\sup_{n\in\Z}\|\om F(2^n\cdot)\|_{H_2^s}<\infty
$$
for some $s>D/2$, then $F(-\Delta)$ is a bounded linear operator on
$L^p(\R^D)$ for all $p\in(1,\infty)$, and one has 
$$
\|F(-\Delta)\|_{L^p\to L^p}
\leq
C_p\Bigl(\sup_{n\in\Z}\|\om F(2^n\cdot)\|_{H_2^s}+|F(0)|\Bigr)\,,
$$
where $C_p$ is a constant not depending on $F$.
\end{theorem}
H{\"o}rmander's multiplier theorem was generalized, on the one hand,
to other spaces than $\R^D$ and, on the other hand, to more general
operators than the Laplacian. 
G.\ Mauceri and S.\ Meda (\cite{MaMe}) and M.\ Christ
(\cite{Christ}) extended the result to homogeneous Laplacians on
stratified nilpotent Lie groups. Further generalizations were
obtained by G.\ Alexopoulos (\cite{Alexo}) who showed in the setting
of connected Lie groups of polynomial volume growth a corresponding
statement for the left invariant sub-Laplacian. This in turn was
extended by W.\ Hebisch (\cite{Hebisch}) to integral operators with
kernels decaying polynomially away from the diagonal.
The results in \cite{DOS} due to X.T.\ Duong, E.M.\ Ouhabaz, and A.\
Sikora marked an important step toward the study of more general
operators. In the abstract framework of spaces of homogeneous type
they investigated non-negative, self-adjoint operators $A$ on
$L^2(X)$ which satisfy {\em pointwise Gaussian estimates}, i.e.\ the 
semigroup $(e^{-tA})_{t>0}$ generated by $-A$ can be represented as 
integral operators
$$
e^{-tA}f(x)=\int_Xp_t(x,y)f(y)\,d\mu(y)
$$
for all $f\in L^2(X)$, $t>0$, $\mu$-a.e.\ $x\in X$ and the kernels
$p_t\colon X\times X\to\C$ enjoy the following pointwise upper bound
\begin{align}\label{GE}
|p_t(x,y)|\leq
C\,\mu(B(x,t^{1/2}))^{-1}\exp\biggl(-b\,\frac{d(x,y)^2}{t}\biggr)
\end{align}
for all $t>0$ and all $x,y\in X$, where $b,C>0$ are constants
independent of $t,x,y$. Under these hypotheses the operator $F(A)$
is of weak type $(1,1)$ whenever $F\colon[0,\infty)\to\C$ is a bounded
Borel function such that $\sup_{n\in\Z}\|\om F(2^n\cdot)\|_{C^s}
<\infty$ for some $s>D/2$ (cf.\ \cite[Theorem 3.1]{DOS}). Consequently,
$F(A)$ is then bounded on $L^p(X)$ for all $p\in(1,\infty)$.

Sometimes it is not clear whether, or even not true that, a
non-negative, self-adjoint operator on $L^2(X)$ admits such Gaussian
bounds and thus the above result would not be applicable. This occurs,
for example, for Schr{\"o}dinger operators with bad potentials
(\cite{Schreieck}) or elliptic operators of higher order with
bounded measurable coefficients (\cite{D97}). Nevertheless, it is
often possible to show a weakened version of pointwise Gaussian estimates,
so-called generalized Gaussian estimates.
\begin{definition}
Let $1\leq p\leq2\leq q\leq\infty$. A non-negative, self-adjoint
operator $A$ on $L^2(X)$ satisfies {\em generalized Gaussian
$(p,q)$-estimates} if there exist constants $b,C>0$ such that
\begin{align}\label{GGE}
\bigl\|\cf_{B(x,t^{1/2})}e^{-tA}\cf_{B(y,t^{1/2})}\bigr\|_{L^p\to L^q}
\leq
C\,\mu(B(x,t^{1/2}))^{-(\frac1{p}-\frac1{q})}
\exp\biggl(-b\,\frac{d(x,y)^2}{t}\biggr)
\end{align}
for all $t>0$ and all $x,y\in X$. In this case, we will use the
shorthand notation GGE$(p,q)$. If $A$ satisfies GGE$(2,2)$, then we
also say that $A$ enjoys {\em Davies-Gaffney estimates}.

Here, $\cf_{E_1}$ denotes the characteristic function of the set
$E_1$ and $\|\cf_{E_1}e^{-tA}\cf_{E_2}\|_{L^p\to L^q}$ is defined
via $\sup_{\|f\|_{L^p}\leq1}\|\cf_{E_1}\cdot
e^{-tA}(\cf_{E_2}f)\|_{L^q}$ for all Borel sets $E_1,E_2\subset X$.
\end{definition}
In the case $(p,q)=(1,\infty)$, this definition covers Gaussian
estimates (cf.\ \cite[Proposition 2.9]{BK1}). It is known that, for
the class of operators $A$ satisfying GGE$(p_0,p_0')$, where
$p_0\in[1,2)$, the interval $[p_0,p_0']$ is, in general, optimal for
the existence of the semigroup $(e^{-tA})_{t>0}$ on $L^p(X)$ for
each $p\in[p_0,p_0']$ (cf.,\ e.g.\ \cite[Theorem 10]{D97}).

In \cite{KunUhl} we show a spectral multiplier
result that covers operators enjoying generalized Gaussian
estimates as well. 
\begin{theorem}\cite[Theorem 5.4]{KunUhl}\label{specmultthm}
Assume that $A$ is a non-negative, self-adjoint operator on $L^2(X)$
satisfying generalized Gaussian $(p_0,p_0')$-estimates for some
$p_0\in[1,2)$. 
Let $p\in(p_0,p_0')$ and $s>D|1/p-1/2|$. Then, for any bounded
Borel function $F\colon[0,\infty)\to\C$ with $\sup_{n\in\Z}\|\om
F(2^n\cdot)\|_{C^s}<\infty$, the operator $F(A)$ is bounded on
$L^p(X)$. More precisely, there exists a constant $C_p>0$ such that
$$
\|F(A)\|_{L^p\to L^p}\leq C_p\Bigl(\sup_{n\in\Z}\|\om
F(2^n\cdot)\|_{C^s}+|F(0)|\Bigr)\,.
$$
\end{theorem}

\begin{remark}
(1) The spectral multiplier result in \cite{DOS} corresponds to the
    case $p_0=1$, i.e. to the case of Gaussian type kernel bounds
    \eqref{GE}.

(2) There is an earlier version of the result due to S. Blunck
    (\cite[Theorem 1.1]{B}) under stronger assumptions on the
    differentiability order $s$ in the H{\"o}rmander condition.

(3) The assertion of Theorem \ref{specmultthm} remains valid for 
    vector valued operators on $L^p(X,\C^n)$.

(4) There are spectral multiplier results for Hardy spaces $H^1_A$
    associated with elliptic second order operators $A$
    (cf.\ \cite{DY,Dzi2}). Via an interpolation argument already used in
    \cite{Kri} this also yields the assertion of Theorem
    \ref{specmultthm}. In fact, this is the idea behind the approach
    we use in \cite{KunUhl}. However, the result in \cite{KunUhl}
    applies to operators satisfying generalized Gaussian estimates of
    \emph{any order}, which means that essential technical tools as
    the finite propagation speed of the wave equation for $A$, which
    one has for second order operators, cannot be used.

(5) Theorem \ref{specmultthm} is formulated with H\"older spaces $C^s$
    in place of Bessel potential spaces $H^s_2$. We refer to the
    discussion in \cite{DOS} where it is pointed out that, in general,
    one cannot replace $C^s$ by $H^s_2$ without additional assumptions.

(6) Very recently, a spectral multiplier theorem under the assumption
    of general Gaussian estimates of second order has been shown in
    \cite{COSY}. The proof does not rely on a result in Hardy spaces.
    It relies on results from \cite{B} and makes even heavier use of
    finite propagation speed, which holds only for second order. 
\end{remark}

\section{The Maxwell operator}

We provide a short overview on the definitions and some
basic properties of the natural function spaces needed for defining
the Maxwell operator. We start with the specification of the
underlying domain. Throughout the whole section, let $\Om$ be a {\em
bounded Lipschitz domain in} $\R^3$, i.e.\ a bounded, connected,
open subset of $\R^3$ with a Lipschitz continuous boundary
$\partial\Om$. This definition allows domains with
corners, but cuts or cusps are excluded. Further, we remark that the
unit exterior normal field $\nu\colon\partial\Om\to\R^3$ can then be
defined almost everywhere on the boundary $\partial\Om$ of $\Om$.

We consider the differential operators divergence $\div$ and
rotation $\rot$ on $L^2(\Om,\C^3)$ in the distributional sense and
introduce the function space
\begin{align}\label{FctSpaceV}
V(\Om):=\bigl\{u\in L^2(\Om,\C^3)\,:\,\div u\in L^2(\Om,\C),
\,\rot u\in L^2(\Om,\C^3),\,\nu\cdot u|_{\partial\Om}=0\bigr\}
\end{align}
equipped with the inner product
$$
(u,v)_{V(\Om)}
:=
(u,v)_{L^2(\Om,\C^3)}+(\div u,\div v)_{L^2(\Om,\C)}+
(\rot u,\rot v)_{L^2(\Om,\C^3)}\,.
$$
Then $V(\Om)$ becomes a Hilbert space which is dense in
$L^2(\Om,\C^3)$. 
Note that the boundary condition of $V(\Om)$ means 
that the exterior normal component vanishes.
In general, $V(\Om)$ is not contained in
$H^1(\Om,\C^3)$ (cf.,\ e.g.\ \cite[p.\ 832]{Amrouche}).
However, under additional assumptions on the
domain $\Om$ the space $V(\Om)$ is continuously embedded into
$H^1(\Om,\C^3)$. For example, this is the case if $\Om$ has a
$C^{1,1}$-boundary or if $\Om$ is convex (cf.,\ e.g.\ \cite[Theorems 2.9
and 2.17]{Amrouche}). Nevertheless, the following statement due to
D.\ Mitrea, M.\ Mitrea, and M.\ Taylor (\cite[p.\ 87]{MMT}) holds
for arbitrary bounded Lipschitz domains $\Om$ in $\R^3$.
\begin{fact}\label{V-H1/2-embedding}
The space $V(\Om)$ is continuously embedded into
$H^{1/2}(\Om,\C^3)$. More precisely, there exists a constant $C>0$
depending only on the boundary $\partial\Om$ and on the diameter
$\mbox{diam}(\Om)$ of $\Om$ such that for every $u\in V(\Om)$
\begin{align*}
\|u\|_{H^{1/2}(\Om,\C^3)}
\leq
C\bigl(\|u\|_{L^2(\Om,\C^3)}+\|\div u\|_{L^2(\Om,\C)}+
\|\rot u\|_{L^2(\Om,\C^3)}\bigr)\,.
\end{align*}
\end{fact}

The definition of the Maxwell operator on
$L^2(\Om,\C^3)$ shall be given in a quite general framework
without any regularity assumptions on the coefficient
matrix. At a first stage, we introduce a form $\a$ with the form
domain $V(\Om)$ and establish generalized Gaussian estimates for the
corresponding semigroup $(e^{-tA_2})_{t>0}$ on $L^2(\Om,\C^3)$ by using Davies'
perturbation method (cf.\ Theorem \ref{GGE-Hodge}). To the best of
our knowledge, this procedure was never elaborated before in this
context. The Maxwell operator is then defined as the restriction of
$A_2$ on the subspace of divergence-free vector fields.

Fix, once and for all, a matrix-valued function $\eps(\cdot)\in
L^\infty(\Om,\C^{3\times3})$ taking values in the set of positive
definite, hermitian matrices. Assume additionally that
$\eps(\cdot)^{-1}\in L^\infty(\Om,\C^{3\times3})$. As immediate
consequences we deduce that, for almost every $x\in\Om$, the matrix
$\eps(x)^{-1}$ is also hermitian and that $\eps(\cdot)^{-1}$
fulfills the following {\em uniform ellipticity condition}
\begin{align}\label{epsell}
\eps(x)^{-1}\xi\cdot\ov\xi\geq\eps_0|\xi|^2
\end{align}
for all $\xi\in\C^3$ and almost all $x\in\Om$, where the constant
$\eps_0>0$ is independent of $\xi$ and $x$. We consider the densely
defined, sesquilinear form
\begin{align*}
\a(u,v):=\int_\Om\eps(\cdot)^{-1}\rot u\cdot\ov{\rot v}\,dx+
\int_\Om\div u~\ov{\div v}\,dx\qquad(u,v\in\dom(\a))
\end{align*}
with the form domain $\dom(\a):=V(\Om)$. Due to the properties of
the coefficient matrix $\eps(\cdot)^{-1}$, the form $\a$ is
continuous and coercive in the sense that there exist constants
$C_1\geq0$, $C_2>0$ such that for all $u\in V(\Om)$
\begin{align}\label{coercive}
\Re\a(u,u)+C_1\|u\|^2_{L^2(\Om,\C^3)}\geq C_2\|u\|^2_{V(\Om)}
\end{align}
(in fact one can take $C_1=C_2=\min\{\eps_0,1\}$). Moreover, it is
symmetric and satisfies $\Re\a(u,u)\ge0$ for all $u\in V(\Om)$. The {\em
operator $A_2$ associated with the form $\a$} is defined via
\begin{center}
$u\in\dom(A_2)$, $A_2u=f$ \quad if and only if \quad $u\in V(\Om)$
and $\a(u,v)=(f,v)_{L^2(\Om,\C^3)}$ for all $v\in V(\Om)$.
\end{center}
Then $A_2$ is self-adjoint and $-A_2$ generates a bounded analytic semigroup
$(e^{-tA_2})_{t>0}$ acting on $L^2(\Om,\C^3)$ (cf., e.g.\ \cite[p.~450]{Dautray}). 

\begin{theorem}\label{GGE-Hodge}
The operator $A_2$ associated with the form $\a$ enjoys generalized
Gaussian $(3/2,3)$-estimates.
\end{theorem}

\begin{proof}
We just have to show that $A_2$ fulfills generalized Gaussian
$(2,3)$-estimates. Thanks to the self-adjointness of $A_2$,
generalized Gaussian $(3/2,2)$-estimates then follow by dualization
and the claimed generalized Gaussian $(3/2,3)$-estimates by
composition and the semigroup law. We divide the proof into several
steps. The first three steps are devoted to the proof of 
Davies-Gaffney estimates for the operator families
$(e^{-tA_2})_{t>0}$, $\{t^{1/2}\div e^{-tA_2}:\,t>0\}$, and
$\{t^{1/2}\rot e^{-tA_2}:\,t>0\}$. In order to derive these bounds,
we will use Davies' perturbation method. It consists in studying
``twisted'' forms
$$
\a_{\vrh\phi}(u,v):=\a(e^{\vrh\phi}u,e^{-\vrh\phi}v)
\qquad(u,v\in V(\Om)),
$$
where $\vrh\in\R$ and $\phi\in\E:=\{\phi\in C_c^\infty(\ov\Om,\R)
\,:\, \|\partial_j\phi\|_\infty\leq1 \mbox{ for all
}j\in\{1,2,3\}\}$. Observe that the multiplication with a function
of the form $e^{\vrh\phi}$ leaves the space $V(\Om)$ invariant and
hence the form $\a_{\vrh\phi}$ is well-defined. In the remaining two
steps we deduce generalized Gaussian $(2,3)$-estimates for $A_2$ by
combining the Davies-Gaffney estimates and the Sobolev embedding
theorem. In the following, we use the shorthand notation
$\|\cdot\|_{p\to q}$ for the norm $\|\cdot\|_{L^p(\Om,\C^3)\to
L^q(\Om,\C^3)}$.

\smallskip
{\bf Step 1:} We claim that for each $\ga\in(0,1)$ there exists a
constant $\om_0\geq0$ such that for all $u\in V(\Om)$, $\vrh\in\R$,
and $\phi\in\E$
\begin{align}\label{Step1-twistform}
\bigl|\a_{\vrh\phi}(u,u)-\a(u,u)\bigr|\leq\gamma\a(u,u)+\om_0\vrh^2\|u\|_2^2\,.
\end{align}

After expanding $\a_{\vrh\phi}(u,u)$ with the help of the product
rules for $\div$ and $\rot$, we get for any $u\in V(\Om)$,
$\vrh\in\R$, and $\phi\in\E$
\begin{align*}
\bigl|\a_{\vrh\phi}(u,u)-\a(u,u)\bigr|
&\leq
|\vrh|\int_\Om|{\eps(\cdot)^{-1}}(\grad\phi\times u)\cdot\ov{\rot u}|\,dx
+|\vrh|\int_\Om|(\grad\phi\cdot u)\,\ov{\div u}|\,dx
\\&\quad
+|\vrh|\int_\Om|\eps(\cdot)^{-1}\rot u\cdot(\grad\phi\times \ov u)|\,dx
+|\vrh|\int_\Om|\div u\,(\grad\phi\cdot\ov u)|\,dx
\\&\quad
+\vrh^2\int_\Om|\eps(\cdot)^{-1}(\grad\phi\times u)\cdot(\grad\phi\times\ov u)|\,dx
+\vrh^2\int_\Om|\grad\phi\cdot u|^2\,dx\,.
\end{align*}
We analyze each of the summands on the right-hand side
separately. Let $\del>0$ to be chosen later. By applying the
Cauchy-Schwarz inequality, by using the elementary inequality
$ab\leq\del a^2+\frac1{4\del}b^2$, which is valid for any real
numbers $a,b$, and by recalling the properties of $\phi$, we can
estimate the first term in the following way
\begin{align*}
&|\vrh|\int_\Om|{\eps(\cdot)^{-1}}(\grad\phi\times u)\cdot\ov{\rot u}|\,dx
\leq
|\vrh|\int_\Om\|\eps(\cdot)^{-1}\|_\infty|\grad\phi|\,|u|\,|\rot u|\,dx
\\&\quad\leq
\|\eps(\cdot)^{-1}\|_\infty\sqrt3\,\|\grad\phi\|_\infty\int_\Om|\rot u|\,|\vrh|\,|u|\,dx
\leq
\sqrt3\,\|\eps(\cdot)^{-1}\|_\infty\Bigl(\del\|\rot u\|_2^2+
\frac1{4\del}\,\vrh^2\|u\|_2^2\Bigr)\,.
\end{align*}
The second term is bounded by
\begin{align*}
&|\vrh|\int_\Om|(\grad\phi\cdot u)\,\ov{\div u}|\,dx
\leq
|\vrh|\int_\Om|\grad\phi|\,|u|\,|\div u|\,dx
\\&\quad\leq
\sqrt3\int_\Om|\vrh|\,|u|\,|\div u|\,dx
\leq
\sqrt3\Bigl(\del\|\div u\|_2^2+\frac1{4\del}\,\vrh^2\|u\|_2^2\Bigr)\,.
\end{align*}
The third term can be treated analogously to the first term
\begin{align*}
|\vrh|\int_\Om|\eps(\cdot)^{-1}\rot u\cdot(\grad\phi\times\ov u)|\,dx
\leq
\sqrt3\,\|\eps(\cdot)^{-1}\|_\infty\Bigl(\del\|\rot u\|_2^2+
\frac1{4\del}\,\vrh^2\|u\|_2^2\Bigr)\,.
\end{align*}
The estimate for the fourth term is prepared in a similar manner as
that for the second term
\begin{align*}
&|\vrh|\int_\Om|\div u\,(\grad\phi\cdot\ov u)|\,dx
\leq
\sqrt3\Bigl(\del\|\div u\|_2^2+\frac1{4\del}\,\vrh^2\|u\|_2^2\Bigr)\,.
\end{align*}
The dealing with the fifth term consists in
\begin{align*}
&\vrh^2\int_\Om|{\eps(\cdot)^{-1}}(\grad\phi\times u)\cdot(\grad\phi\times\ov u)|\,dx
\leq
\vrh^2\int_\Om|{\eps(\cdot)^{-1}}(\grad\phi\times u)|\,|\grad\phi\times\ov u|\,dx
\\&\quad\leq
\vrh^2\int_\Om\sqrt3\,\|\eps(\cdot)^{-1}\|_\infty\,|u|\,\sqrt3\,|u|\,dx
=
3\|\eps(\cdot)^{-1}\|_\infty\vrh^2\|u\|_2^2\,,
\end{align*}
whereas the sixth term is bounded by
\begin{align*}
\vrh^2\int_\Om|\grad\phi\cdot u|^2\,dx
\leq
\vrh^2\int_\Om|\grad\phi|^2|u|^2\,dx
\leq
3\vrh^2\|u\|_2^2\,.
\end{align*}
By putting all these estimates together, we finally end up with
\begin{align*}
\bigl|\a_{\vrh\phi}(u,u)-\a(u,u)\bigr|
&\leq
\bigl(2\sqrt3\,\|\eps(\cdot)^{-1}\|_\infty+2\sqrt3\bigr)\,\del\,\bigl(\|\rot u\|_2^2+\|\div u\|_2^2\bigr)
\\&\quad\quad+
\Bigl(\bigl(2\sqrt3\,\|\eps(\cdot)^{-1}\|_\infty+2\sqrt3\bigr)\frac1{4\del}+
3\|\eps(\cdot)^{-1}\|_\infty+3\Bigr)\vrh^2\|u\|_2^2\,.
\end{align*}
The ellipticity property (\ref{epsell}) of the coefficient matrix
$\eps(\cdot)^{-1}$ yields for each $u\in V(\Om)$
\begin{align}\label{Step1-ineq1}
\a(u,u)\geq\min\{\eps_0,1\}\bigl(\|\rot u\|_2^2+\|\div u\|_2^2\bigr)\,.
\end{align}
Now let $\ga\in(0,1)$ be arbitrary. Take $\del>0$ such that
$\ga=(2\sqrt3\,\|\eps(\cdot)^{-1}\|_\infty+2\sqrt3)\,\del/\min\{\eps_0,1\}$.
Then we deduce for each $u\in V(\Om)$, $\vrh\in\R$, and $\phi\in\E$
\begin{align*}
\bigl|\a_{\vrh\phi}(u,u)-\a(u,u)\bigr|
&\leq
\ga\a(u,u)+\om_0\vrh^2\|u\|_2^2
\end{align*}
with some constant $\om_0\geq0$ depending exclusively on $\ga$,
$\eps_0$, $\|\eps(\cdot)^{-1}\|_\infty$. This shows
(\ref{Step1-twistform}).

\smallskip
{\bf Step 2:} Due to (\ref{Step1-twistform}), if $\om>\om_0$, we can
write for any $u\in V(\Om)$, $\vrh\in\R$, and $\phi\in\E$
\begin{align*}
\Re\a_{\vrh\phi}(u,u)
\geq
\a(u,u)-\bigl|\a(u,u)-\a_{\vrh\phi}(u,u)\bigr|
\geq
(1-\ga)\a(u,u)-\om\vrh^2\|u\|_2^2\,.
\end{align*}
By recalling (\ref{Step1-ineq1}), we thus have shown that the form
$\a_{\vrh\phi\om}:=\a_{\vrh\phi}+\om\vrh^2$ is coercive in the sense
of (\ref{coercive}) with $C_1=C_2=(1-\ga)\min\{\eps_0,1\}$. This
entails that the operator $A_{\vrh\phi\om}$ associated with the form
$\a_{\vrh\phi\om}$ is sectorial of some angle
$\theta_0\in(0,\pi/2)$. Therefore, $-A_{\vrh\phi\om}$ generates a
bounded analytic semigroup $(e^{-tA_{\vrh\phi\om}})_{t>0}$ on
$L^2(\Om,\C^3)$ and additionally
\begin{align}\label{Step2-SG}
\bigl\|e^{-zA_{\vrh\phi\om}}\bigr\|_{2\to2}\leq1
\end{align}
for all $z\in\C\setminus\{0\}$ with $|\arg z|\leq\theta_0$. In view
of \cite[Lemma 3.2]{LisSobVog}, this yields for any $\vrh\in\R$,
$\phi\in\E$, and $z\in\C\setminus\{0\}$ with $|\arg z|\leq\theta_0$
\begin{align}\label{Step2-SG-weight}
\bigl\|e^{-\vrh\phi}e^{-zA_2}e^{\vrh\phi}\bigr\|_{2\to2}\leq e^{\om\vrh^2\Re z}
\end{align}
and thus, by a similar reasoning as in the proof of
\cite[Proposition 8.22]{Levico}, the operator $A_2$ satisfies 
Davies-Gaffney estimates. Additionally, we have for each $\vrh\in\R$,
$\phi\in\E$, and $t>0$
\begin{align*}
\bigl\|A_{\vrh\phi\om}e^{-tA_{\vrh\phi\om}}\bigr\|_{2\to2}
\leq
\frac1{t\sin\theta_0}\,.
\end{align*}
Indeed, this estimate follows easily from Cauchy's formula and (\ref{Step2-SG})
\begin{align*}
\bigl\|A_{\vrh\phi\om}e^{-tA_{\vrh\phi\om}}\bigr\|_{2\to2}\nonumber
&=
\bigl\|\frac1{2\pi i}\int_{|z-t|=t\sin\theta_0}
\frac1{(z-t)^2}\,e^{-zA_{\vrh\phi\om}}\,dz\bigr\|_{2\to2}
\\&\leq
\frac1{2\pi}\,2\pi t\sin\theta_0\,\frac1{(t\sin\theta_0)^2}\,
=
\frac1{t\sin\theta_0}\,.
\end{align*}

\smallskip
{\bf Step 3:} Our next task consists in verifying Davies-Gaffney
estimates for the operator families $\{t^{1/2}\div
e^{-tA_2}:\,t>0\}$ and $\{t^{1/2}\rot e^{-tA_2}:\,t>0\}$.

For arbitrary $f\in C_c^\infty(\Om,\C^3)$, $\vrh\in\R$, $\phi\in\E$,
$\om>\om_0$, and $t>0$ define $v(t):=e^{-tA_{\vrh\phi\om}}f$. Then
$v(t)$ belongs to $\dom(A_{\vrh\phi\om})$ and, due to
(\ref{Step1-ineq1}) and the estimates in Step 2, we obtain
\begin{align*}
&\bigl\|\rot v(t)\bigr\|_2^2+\bigl\|\div v(t)\bigr\|_2^2
\leq
\frac1{\min\{\eps_0,1\}}\,\a(v(t),v(t))
\leq
\frac1{(1-\ga)\min\{\eps_0,1\}}\,\Re\a_{\vrh\phi\om}(v(t),v(t))
\\&\quad\leq
\frac1{(1-\ga)\min\{\eps_0,1\}}\,\bigl|(A_{\vrh\phi\om}v(t),v(t))_{L^2(\Om,\C^3)}\bigr|
\leq
\frac1{(1-\ga)\min\{\eps_0,1\}}\,\|A_{\vrh\phi\om}v(t)\|_2\|v(t)\|_2
\\&\quad\leq
\frac1{(1-\ga)\min\{\eps_0,1\}\sin\theta_0}\, t^{-1}\|f\|_2^2\,.
\end{align*}
As the space of test functions $C_c^\infty(\Om,\C^3)$ is dense in
$L^2(\Om,\C^3)$, we conclude that
\begin{align}
\bigl\|\div e^{-tA_{\vrh\phi}}\bigr\|_{2\to2}
\leq
\frac1{\sqrt{(1-\ga)\min\{\eps_0,1\}\sin\theta_0}}\,t^{-1/2}e^{\om\vrh^2t}
\nonumber \intertext{and}
\bigl\|\rot e^{-tA_{\vrh\phi}}\bigr\|_{2\to2}
\leq
\frac1{\sqrt{(1-\ga)\min\{\eps_0,1\}\sin\theta_0}}\,t^{-1/2}e^{\om\vrh^2t}
\label{Step3-rot}
\end{align}
for all $\vrh\in\R$, $\phi\in\E$, $\om>\om_0$, and $t>0$.

In order to obtain weighted norm estimates for $t^{1/2}\rot
e^{-tA_2}$, we have to interchange $\rot\!$ and multiplication by
$e^{-\vrh\phi}$. 
To this end, we represent $e^{-\vrh\phi}\rot h$ in terms of
$\rot(e^{-\vrh\phi}h)$ and apply this representation to
$h:=e^{-tA_2}e^{\vrh\phi}f$. 
By using the product rule for $\rot$ we obtain
\begin{align*}
e^{-\vrh\phi}\rot h
=
\rot(e^{-\vrh\phi}h)+\vrh\grad\phi\times(e^{-\vrh\phi}h)\,.
\end{align*}
The $L^2$-norm of the first term on the right-hand side can be
estimated by (\ref{Step3-rot}), whereas for the second term we use
$\|\grad\phi\|_\infty\leq\sqrt3$, the elementary fact that $|\vrh|
\leq C_\del t^{-1/2}e^{\del\vrh^2t}$ for arbitrary $\del>0$ and some
constant $C_\del>0$ depending only on $\del$, and
(\ref{Step2-SG-weight})
\begin{align*}
\bigl\|e^{-\vrh\phi}\rot e^{-tA_2}e^{\vrh\phi}f\bigr\|_2
&=
\bigl\|e^{-\vrh\phi}\rot h\bigr\|_2
\leq
\bigl\|\rot(e^{-\vrh\phi}h)\bigr\|_2+|\vrh|\,\|\grad\phi\|_\infty\|e^{-\vrh\phi}h\|_2
\\&\leqC
t^{-1/2}e^{(\om+\del)\vrh^2t}\|f\|_2
\end{align*}
which yields
\begin{align*}
\bigl\|e^{-\vrh\phi}t^{1/2}\rot e^{-tA_2}e^{\vrh\phi}\bigr\|_{2\to2}
\leqC
e^{(\om+\del)\vrh^2t}\,.
\end{align*}
By adapting the arguments given in the proof of \cite[Proposition 8.22]{Levico}
we see that 
the family of operators $\{t^{1/2}\rot e^{-tA_2}:\,t>0\}$ satisfies
Davies-Gaffney estimates. A similar reasoning shows that $\{t^{1/2}\div
e^{-tA_2}:\,t>0\}$ enjoys the same property.

\smallskip
{\bf Step 4:} Let $\Om_0$ be a bounded Lipschitz domain in $\R^3$.
In view of Fact \ref{V-H1/2-embedding} and the Sobolev embedding
$H^{1/2}(\Om_0,\C^3)\hookrightarrow L^{p^*}\!(\Om_0,\C^3)$ for
$p^*:=\frac{3\cdot2}{3-1}=3$, we find a constant $C>0$ depending only
on $\partial\Om_0$ and $\mbox{diam}(\Om_0)$ such that, for every
$u\in V(\Om_0)$,
\begin{align}\label{V-Einbettung-Sobolev}
\|u\|_{L^3(\Om_0,\C^3)}
\leq
C\bigl(\|u\|_{L^2(\Om_0,\C^3)}+\|\div u\|_{L^2(\Om_0,\C)}
+
\|\rot u\|_{L^2(\Om_0,\C^3)}\bigr)\,.
\end{align}
With the help of the rescaling procedure used in \cite[p.\
3145]{MiMo}, we get, for all $w\in V(\Om_0)$,
\begin{align}\label{V-Einbettung}
&\|w\|_{L^3(\Om_0,\C^3)}
\leq C
R^{-1/2}\bigl(\|w\|_{L^2(\Om_0,\C^3)}+R\,\|\div w\|_{L^2(\Om_0,\C)}
+
R\,\|\rot w\|_{L^2(\Om_0,\C^3)}\bigr)\,,
\end{align}
where $R:=\mbox{diam}(\Om_0)$ and the constant $C$ depends
exclusively on the Lipschitz character of $\Om_0$.

\smallskip
{\bf Step 5:} The desired generalized Gaussian $(2,3)$-estimates for
$A_2$ follow by combining the Davies-Gaffney estimates from Step 2
and 3 with inequality (\ref{V-Einbettung}). A similar
reasoning had been applied in \cite[Section 5]{MiMo}.

Let $t>0$, $x,y\in\Om$, and let $f\in C_c^\infty(\Om,\C^3)$ with $\supp
f\subset B(y,t^{1/2})$ be arbitrary. Put $\Om_0:=B(x,2t^{1/2})
\subset\Om$ and choose a cut-off function $\eta\in
C_c^\infty(\Om_0,\R)$ such that
\begin{align*}
0\leq\eta\leq1\,,\qquad
\eta=1\quad\mbox{on }B(x,t^{1/2})\,,\qquad\mbox{and}\qquad
\|\grad\eta\|_\infty\leq t^{-1/2}\,.
\end{align*}
First, we remark that
\begin{align*}
\bigl\|\div(\eta e^{-tA_2}f)\bigr\|_{L^2(\Om_0,\C)}
&\leq
\bigl\|\eta\,\div(e^{-tA_2}f)\bigr\|_{L^2(\Om_0,\C)}
+
\bigl\|\grad\eta\cdot e^{-tA_2}f\bigr\|_{L^2(\Om_0,\C)}
\\&\leqC
\bigl\|\div(e^{-tA_2}f)\bigr\|_{L^2(\Om_0,\C)}
+
t^{-1/2}\bigl\|e^{-tA_2}f\bigr\|_{L^2(\Om_0,\C^3)}
\end{align*}
and similarly
\begin{align*}
\bigl\|\rot(\eta e^{-tA_2}f)\bigr\|_{L^2(\Om_0,\C^3)}
&\leqC
\bigl\|\rot(e^{-tA_2}f)\bigr\|_{L^2(\Om_0,\C^3)}
+t^{-1/2}\bigl\|e^{-tA_2}f\bigr\|_{L^2(\Om_0,\C^3)}\,.
\end{align*}
Since $\nu\cdot(\eta e^{-tA_2}f)|_{\partial\Om_0}=0$ and the
Lipschitz character of $\Om_0$ is controlled by that of $\Om$, we
may use (\ref{V-Einbettung}) and arrive at
\begin{align}\label{Step5}
&\bigl\|e^{-tA_2}f\bigr\|_{L^3(B(x,t^{1/2}),\C^3)}
\leq\nonumber
\bigl\|\eta e^{-tA_2}f\bigr\|_{L^3(\Om_0,\C^3)}
\\&\quad\leqC\nonumber
t^{-1/4}\bigl(\bigl\|\eta e^{-tA_2}f\bigr\|_{L^2(\Om_0,\C^3)}
+t^{1/2}\bigl\|\div(\eta e^{-tA_2}f)\bigr\|_{L^2(\Om_0,\C)}
+t^{1/2}\bigl\|\rot(\eta e^{-tA_2}f)\bigr\|_{L^2(\Om_0,\C^3)}\bigr)
\\&\quad\leqC\nonumber
t^{-1/4}\bigl(3\bigl\|e^{-tA_2}f\bigr\|_{L^2(\Om_0,\C^3)}
+t^{1/2}\bigl\|\div(e^{-tA_2}f)\bigr\|_{L^2(\Om_0,\C)}
+t^{1/2}\bigl\|\rot(e^{-tA_2}f)\bigr\|_{L^2(\Om_0,\C^3)}\bigr)
\\&\quad\leqC
t^{-\frac32(\frac12-\frac13)}
\exp\biggl(-b\,\frac{|x-y|^2}{t}\biggr)\|f\|_{L^2(B(y,t^{1/2}),\C^3)}\,,
\end{align}
where the implicit constants are independent of $f,t,x,y$ and the
last inequality is due to the Davies-Gaffney estimates for
$(e^{-tA_2})_{t>0}$, $\{t^{1/2}\div e^{-tA_2}:\,t>0\}$, and
$\{t^{1/2}\rot e^{-tA_2}:\,t>0\}$. Finally, by density, we deduce
generalized Gaussian $(2,3)$-estimates for $A_2$.
\end{proof}

As noted at the beginning of this section, $V(\Om)$ enjoys better
embedding properties if $\Om$ is convex or if its boundary is of
class $C^{1,1}$. In these cases the space $V(\Om)$ continuously
embeds into $H^1(\Om,\C^3)$ which in turn continuously embeds
into $L^6(\Om,\C^3)$. Hence, in this situation one can take the
$L^6(\Om,\C^3)$-norm on the left-hand side of
(\ref{V-Einbettung-Sobolev}). Observe that this automatically gives
the desired exponent of $t$ in (\ref{Step5}) (cf., e.g.\ \cite[proof
of Theorem 3.1]{PeerJFA}) and thus the rescaling argument in Step 4
would not be needed. Summing up, the following statement holds.
\begin{coro}\label{CoroGGE}
In the situation of Theorem \ref{GGE-Hodge} suppose additionally
that the domain $\Om$ is convex or has a $C^{1,1}$-boundary. Then
the operator $A_2$ associated with the form $\a$ satisfies
generalized Gaussian $(6/5,6)$-estimates.
\end{coro}

Since $A_2$ satisfies generalized Gaussian $(p_0,p_0')$-estimates
for some $p_0\in[1,3/2]$, the semigroup generated by $-A_2$ can be
extended to a bounded analytic semigroup on $L^p(\Om,\C^3)$ for
every $p\in[p_0,p_0']$ with $p\ne\infty$. For the rest of this
section, we denote by $-A_p$ its generator.

\medskip
In order to introduce the Maxwell operator, we first recall some
basic facts concerning the Helmholtz decomposition in
$L^p(\Om,\C^3)$. 
For $p\in(1,\infty)$ define the space of divergence-free vector fields
\begin{align*}
L^p_{\si}(\Om)
:=\bigl\{v\in L^p(\Om,\C^3)\,:\,\div v=0,\,\nu\cdot v|_{\partial\Om}=0\bigr\}
\end{align*}
and the space of gradients
\begin{align*}
G^p(\Om):=\bigl\{\grad g\,:\,g\in W^1_p(\Om,\C)\bigr\}\,.
\end{align*}
Then both are closed subspaces of $L^p(\Om,\C^3)$. In the case $p=2$
the corresponding orthogonal projection $\P_2$ from $L^2(\Om,\C^3)$
onto $L^2_{\si}(\Om)$ is called the {\em Helmholtz projection}.
E.\ Fabes, O.\ Mendez, and M.\ Mitrea established a {\em Helmholtz
decomposition in} $L^p(\Om,\C^3)$ which reads as follows.
\begin{fact} \cite[Theorems 11.1 and 12.2]{FMM}
For every bounded Lipschitz domain $\Om$ in $\R^3$ there exists
$\eps>0$ such that $\P_2$ extends to a bounded linear operator
$\P_p$ from $L^p(\Om,\C^3)$ onto $L^p_{\si}(\Om)$ for all
$p\in(3/2-\eps,3+\eps)$. In this range, one has an {\em
$L^p$-Helmholtz decomposition}
\begin{align}\label{LpHHdecomp}
L^p(\Om,\C^3)=L^p_{\si}(\Om)\oplus G^p(\Om)
\end{align}
as a topological direct sum. The operator $\P_p$ is then called the {\em
$L^p$-Helmholtz projection}.

In the class of bounded Lipschitz domains, this result is sharp in
the sense that, for any $p\notin[3/2,3]$, there is a bounded
Lipschitz domain $\Om\subset\R^3$ for which the $L^p$-Helmholtz
decomposition (\ref{LpHHdecomp}) fails.

If, however, $\Om$ has a regular boundary $\partial\Om\in C^1$, then
the result is even true for all $p\in(1,\infty)$.
\end{fact}

For convenience, we introduce the following abbreviation.
\begin{notation}
We denote by $I_\Om$ the largest subinterval of the real line
containing $2$ such that
for each $p\in I_\Om$ the semigroup $(e^{-tA_2})_{t>0}$ extends to a
bounded analytic semigroup on $L^p(\Om,\C^3)$ and that there exists
an $L^p$-Helmholtz decomposition.
\end{notation}
In view of the foregoing statements, the length of $I_\Om$ is
intimately related to regularity properties of the boundary $\partial\Om$
and the interval $[3/2,3]$ is always contained in $I_\Om$.

As we shall see immediately, the operators $A_2$ and $\P_2$ are
commuting. This relies on the fact that $\P_2$ leaves the domain
$V(\Om)$ of the form $\a$ invariant which is essentially due to the
boundary condition of $V(\Om)$. We remark that this property stands
in contrast to the situation of Dirichlet boundary
conditions ($\nu\cdot v|_{\partial\Om}=0$ and $\nu\times
v|_{\partial\Om}=0$). The latter is implicitly mentioned in
\cite[Chapter 4]{constantin}.

\begin{lemma}\label{HHPro-komm}
For any $p\in I_\Om$, the operator $A_p$ and the Helmholtz
projection $\P_p$ are commuting, i.e.\ $\P_p(\dom(A_p))$ is
contained in $\dom(A_p)$ and it holds, for all $u\in\dom(A_p)$,
$$
\P_pA_pu=A_p\P_pu\,.
$$
\end{lemma}
\begin{proof}
At first, we treat the case $p=2$. The statement for arbitrary $p\in
I_\Om$ then follows by density and consistency.

We claim that $\P_2\colon V(\Om)\to V(\Om)$. Indeed, let $u\in
V(\Om)$. By definition of $\P_2$, it is evident that $\div(\P_2u)=0$
as well as $\nu\cdot(\P_2u)|_{\partial\Om}=0$. In order to check
$\rot(\P_2u)\in L^2(\Om,\C^3)$, we write $\P_2u=u-\grad g$ for some
$g\in W^1_2(\Om,\C)$ and note that it suffices to show $\rot(\grad
g)=0$. This can be easily verified via the distributional
definitions of $\rot$ and $\grad$ which transfer the assertion to
the level of test functions where it is elementary.
In particular, we have just computed $\rot(\P_2u)=\rot u$ for every
$u\in V(\Om)$. 

Now consider $u\in\dom(A_2)$. We get for each $v\in V(\Om)$
\begin{align*}
(\P_2A_2u,v)_{L^2(\Om,\C^3)}
=
(A_2u,\P_2v)_{L^2(\Om,\C^3)}
=
\a(u,\P_2v)
=
\a(\P_2u,v)\,,
\end{align*}
where the last equality is obtained with the help of
$\rot(\P_2u)=\rot u$. This means that $\P_2u\in\dom(A_2)$
and $\P_2A_2u=A_2\P_2u$.

Let $p\in I_\Om$. Observe that $A_p$ and $\P_p$ are commuting if and
only if resolvents of $A_p$ commute with $\P_p$ on $L^p(\Om,\C^3)$.
In particular, we have seen above that
$\P_p(\la+A_p)^{-1}=(\la+A_p)^{-1}\P_p$ on $L^p(\Om,\C^3)\cap
L^2(\Om,\C^3)$ for any $\la\in\C$ with $\Re\la>0$. Since $-A_2$ as
well as $-A_p$ are generators of bounded analytic semigroups, their
resolvent sets include the right-half complex plane and their
resolvents are consistent. Hence, by the density of
$L^p(\Om,\C^3)\cap L^2(\Om,\C^3)$ in $L^p(\Om,\C^3)$ and by the
boundedness of resolvent operators, the equality
$\P_p(\la+A_p)^{-1}=(\la+A_p)^{-1}\P_p$ extends to the whole space
$L^p(\Om,\C^3)$. This yields the lemma.
\end{proof}

Now we are prepared to introduce the Maxwell operator.

\begin{definition}\label{DefMaxOp}
For $p\in I_\Om$ we define the {\em Maxwell operator} $M_p$ on
$L^p_\sigma(\Om)$ by setting
\begin{align*}
\dom(M_p)&:=\P_p\dom(A_p)=\dom(A_p)\cap L^p_\si(\Om)\,,\\
M_pu&:=A_pu\qquad\mbox{for }u\in\dom(M_p)\,.
\end{align*}
\end{definition}

Since $A_2$ satisfies generalized Gaussian $(3/2,3)$-estimates (cf.\
Theorem \ref{GGE-Hodge}), Theorem \ref{specmultthm} yields the
following result.
\begin{theorem}\label{SpecMultHodgeLaplace}
Let $p\in(3/2,3)$. Suppose that $s>3|1/p-1/2|$. Then, for every
bounded Borel function $F\colon[0,\infty)\to\C$ with
$\sup_{n\in\Z}\|\om F(2^n\cdot)\|_{C^s}<\infty$, the operator
$F(A_2)$ is bounded on $L^p(\Om,\C^3)$ and there exists a constant
$C_p>0$ such that
$$
\|F(A_2)\|_{L^p(\Om,\C^3)\to L^p(\Om,\C^3)}
\leq
C_p\Bigl(\sup_{n\in\Z}\|\om F(2^n\cdot)\|_{C^s}+|F(0)|\Bigr)\,.
$$
\end{theorem}
As $A_p$ and $\P_p$ are commuting, the functional calculus for $A_2$
on $L^p(\Om,\C^3)$ and the Helm\-holtz projection $\P_p$ are
commuting as well. Therefore, we deduce a spectral multiplier
theorem for the Maxwell operator by restricting $F(A_2)$ to the
space of divergence-free vector fields.
\begin{theorem}\label{SpecMultMaxwell}
Let $p\in(3/2,3)$. Suppose that $s>3|1/p-1/2|$. Then, for every
bounded Borel function $F\colon[0,\infty)\to\C$ with
$\sup_{n\in\Z}\|\om F(2^n\cdot)\|_{C^s}<\infty$, the operator
$F(M_2)$ is bounded on $L^p_\si(\Om)$ and there exists a constant
$C_p>0$ such that
$$
\|F(M_2)\|_{L^p_\si(\Om)\to L^p_\si(\Om)}
\leq
C_p\Bigl(\sup_{n\in\Z}\|\om F(2^n\cdot)\|_{C^s}+|F(0)|\Bigr)\,.
$$
\end{theorem}

\begin{remark}
(1) If $\Om$ is convex or has a $C^{1,1}$-boundary, the assertions of
    Theorems \ref{SpecMultHodgeLaplace} and \ref{SpecMultMaxwell}
    even hold for any $p\in(6/5,6)$ because $A_2$ then satisfies
    generalized Gaussian $(6/5,6)$-estimates (cf.\ Corollary 
    \ref{CoroGGE}).

(2) The situation on the whole space $\Om=\R^3$ is more comfortable
    because no boundary terms occur. In particular, the form $\a$ is 
    better suited concerning partial integration. Note that the range 
    of values $p_0\in[1,2)$ for which our method gives generalized 
    Gaussian $(p_0,p_0')$-estimates for $A_2$ then depends only on the
    regularity of the coefficient matrix $\eps(\cdot)$. In the case of
    smooth coefficients one can even prove pointwise Gaussian 
    estimates for $A_2$.
\end{remark}

\section{The Stokes operator with Hodge boundary conditions}

In this section we show that the spectral multiplier result,
presented in Theorem \ref{specmultthm} above, also holds for the
Stokes operator $A$ with Hodge boundary conditions. Our argument is
based on off-diagonal norm estimates for the resolvents of the
Hodge-Laplacian which were recently established by M.\ Mitrea and
S.\ Monniaux (\cite{MiMo}). We verify that bounds of this type 
entail the validity of generalized Gaussian estimates for the
Hodge-Laplacian so that Theorem \ref{specmultthm} can be applied.
Since $A$ is the restriction of the Hodge-Laplacian to the space of
divergence-free vector fields and the Hodge-Laplacian and the Helmholtz
projection are commuting, we obtain in a similar way as in the
foregoing section a spectral multiplier theorem for the Stokes
operator with Hodge boundary conditions.

\smallskip
Let $\Om\subset\R^3$ be a bounded Lipschitz domain and $V(\Om)$
denote the function space introduced in (\ref{FctSpaceV}). At
first, we recall the definition of the {\em Hodge-Laplacian} $B$
which is the operator associated with the densely defined,
sesquilinear, symmetric form
\begin{align*}
\b(u,v):=\int_\Om\rot u\cdot\ov{\rot v}\,dx+
\int_\Om\div u~\ov{\div v}\,dx\qquad(u,v\in V(\Om)).
\end{align*}
Then $B$ is self-adjoint, invertible, and $-B$ generates an analytic
semigroup on $L^2(\Om,\C^3)$. According to \cite[(3.17) and
(3.18)]{MiMo}, the Hodge-Laplacian $B$ can be characterized by
\begin{align*}
\dom(B)&=\bigl\{u\in V(\Om)\,:\,\rot\rot u\in L^2(\Om,\C^3),\,\div u\in H^1(\Om,\C),\,
\nu\times\rot u|_{\partial\Om}=0\bigr\}\,,\\
Bu&=-\Delta u\qquad\mbox{for }u\in\dom(B)\,.
\end{align*}

\begin{definition}
The {\em Stokes operator $A$ with Hodge boundary conditions} on
$L^2_\si(\Om)$ is defined via $A:=\P_2B$ with the domain
$\dom(A):=\P_2\dom(B)$.
\end{definition}

Starting from norm estimates of annular type on $L^p(\Om,\C^3)$ with
$p=2$ for resolvents of the Hodge-Laplacian $B$, M.\ Mitrea and S.\
Monniaux developed an iterative bootstrap argument (\cite[Lemma
5.1]{MiMo}) that allows to incrementally increase the value of $p$
to $p^*:=\frac32\,p$ (due to Sobolev embeddings), as long
as $p<q_\Om$, where $q_\Om$ denotes the critical index for the
well-posedness of the Poisson type problem for the Hodge-Laplacian
(\cite[(1.9)]{MiMo}). In the present situation of a bounded
Lipschitz domain $\Om$ in $\R^3$, it is known that $q_\Om>3$ (cf.\
\cite{Mitrea-Hodge}). M.\ Mitrea and S.\ Monniaux (\cite[Section
6]{MiMo}) showed that for any $\theta\in(0,\pi)$ there exist
$q\in(3,\infty]$ and constants $b,C>0$ such that for all $j\in\N$,
$x\in\Om$, and $\la\in\C\setminus\{0\}$ with $|\arg\la|<\pi-\theta$
\begin{align}\label{ODResHodge}
\bigl\|\cf_{B(x,|\la|^{-1/2})}\la(\la+B)^{-1}
\cf_{B(x,2^{j+1}|\la|^{-1/2})\setminus B(x,2^{j-1}|\la|^{-1/2})}
\bigr\|_{L^2(\Om,\C^3)\to L^q(\Om,\C^3)}
\leq
C\,|\la|^{\frac 32(\frac12-\frac1q)}\,e^{-b\,2^j}\,.
\end{align}
As we shall see, in the Euclidean setting the validity of those
estimates for resolvent operators ensures generalized Gaussian
$(2,q)$-estimates for the semigroup operators. Since an analytic
semigroup $(e^{-tL})_{t>0}$ and resolvents of its generator $-L$ are
intimately related via integral representations, we obtain a
nearly equivalent formulation of generalized Gaussian estimates if
we replace in the two-ball estimate (\ref{GGE}) the semigroup
operators with resolvent operators of the form $\la(\la+L)^{-1}$ for
$\la\in\rho(-L)$. To be precise, the transfer from resolvent
operators to semigroup operators and vice versa reads as follows.

\begin{lemma}\label{GGE-SG-Res}
Let $\Om\subset\R^D$ be a Borel set, $n\in\N$, and $L$ a
non-negative, self-adjoint operator on $L^2(\Om,\C^n)$. Assume that
$1\leq p\leq2\leq q\leq\infty$ and $m\geq2$ with $D/m(1/p-1/q)<1$.
\begin{compactenum}[\bf a)]
\item
Fix $\theta\in(0,\pi/2)$ and suppose that there exist constants
$b,C>0$ such that for all $x,y\in\Om$ and all
$\la\in\C\setminus\{0\}$ with $|\arg\la|<\pi-\theta$
\begin{align}\label{GGE-SG-Res-a}
&\bigl\|\cf_{B(x,|\la|^{-1/m})}\la(\la+L)^{-1}
\cf_{B(y,|\la|^{-1/m})}\bigr\|_{L^p(\Om,\C^n)\to L^q(\Om,\C^n)}
\leq C\,
|\la|^{\frac Dm(\frac1p-\frac1q)}e^{-b\,|\la|^{1/m}|x-y|}\,.
\end{align}
Then there are constants $b',C'>0$ such that the semigroup
operators satisfy
\begin{align}\label{GGE-SG-Res-b}
\bigl\|\cf_{B(x,t^{1/m})}e^{-tL}\cf_{B(y,t^{1/m})}\bigr\|_{L^p(\Om,\C^n)\to L^q(\Om,\C^n)}
\leq C'\,
t^{-\frac Dm(\frac1p-\frac1q)}\exp\Biggl(
-b'\biggl(\frac{|x-y|}{t^{1/m}}\biggr)^{\frac m{m-1}}\Biggr)
\end{align}
for any $t>0$ and any $x,y\in\Om$.

\item
Suppose that there exist constants $b,C>0$ such that for all $t>0$
and all $x,y\in\Om$
\begin{align*}
\bigl\|\cf_{B(x,t^{1/m})}e^{-tL}\cf_{B(y,t^{1/m})}\bigr\|_{L^p(\Om,\C^n)\to L^q(\Om,\C^n)}
\leq C\,
t^{-\frac Dm(\frac1p-\frac1q)}\exp\Biggl(
-b\biggl(\frac{|x-y|}{t^{1/m}}\biggr)^{\frac m{m-1}}\Biggr)\,.
\end{align*}
Then for any $\theta\in(0,\pi/2)$ there are constants $b',C'>0$ such
that for all $x,y\in\Om$ and all $\la\in\C\setminus\{0\}$ with
$|\arg\la|<\theta$
\begin{align*}
&\bigl\|\cf_{B(x,|\la|^{-1/m})}\la(\la+L)^{-1}
\cf_{B(y,|\la|^{-1/m})}\bigr\|_{L^p(\Om,\C^n)\to L^q(\Om,\C^n)}
\leq C'\,
|\la|^{\frac Dm(\frac1p-\frac1q)}e^{-b'\,|\la|^{1/m}|x-y|}\,.
\end{align*}
\end{compactenum}
\end{lemma}

In view of (\ref{ODResHodge}) and \cite[Proposition 2.1]{BKLeg},
Lemma \ref{GGE-SG-Res} ensures the validity of generalized Gaussian
$(2,q)$-estimates for the Hodge-Laplacian for some $q\in(3,\infty]$.
Similar as in the previous section, Theorem \ref{specmultthm}
entails the boundedness of spectral multipliers, at first for the
Hodge-Laplacian $B$ and then, by restriction, for the Stokes operator
$A$ with Hodge boundary conditions because the Hodge-Laplacian and
the Helmholtz projection are commuting (cf.\ Lemma \ref{HHPro-komm}
or \cite[Lemma 3.7]{MiMo}). This leads to the following statement.

\begin{theorem}\label{SpecMultStokes}
Assume that (\ref{ODResHodge}) holds for some $q\in(3,\infty]$ and
that there is an $L^q$-Helmholtz decomposition. Fix $p\in(q',q)$ and
take $s>3|1/p-1/2|$. Then, for every bounded Borel function
$F\colon[0,\infty)\to\C$ with $\sup_{n\in\Z}\|\om
F(2^n\cdot)\|_{C^s}<\infty$, the operator $F(A)$ is bounded on
$L^p_\si(\Om)$ and there exists a constant $C>0$ such that
$$
\|F(A)\|_{L^p_\si(\Om)\to L^p_\si(\Om)}
\leq
C\Bigl(\sup_{n\in\Z}\|\om F(2^n\cdot)\|_{C^s}+|F(0)|\Bigr)\,.
$$
\end{theorem}

\begin{proof}[Proof of Lemma \ref{GGE-SG-Res}]
As noted in \cite[pp.\ 934-935]{BK2}, one can assume that
$\Om=\R^D$. Otherwise, instead of an operator $T\colon
L^p(\Om,\C^n)\to L^q(\Om,\C^n)$, one considers the extended operator
$\wt T\colon L^p(\R^D,\C^n)$ $\to L^q(\R^D,\C^n)$ defined by
$$
\wt Tu(x):=\left\{
\begin{array}{cl}
T(\cf_\Om u)(x)&\mbox{for }x\in\Om\\
0&\mbox{for }x\notin\Om
\end{array}
\right.\qquad(u\in L^p(\R^D,\C^n),x\in\R^D).
$$
Then it is straightforward to check that $\|\wt
T\|_{L^p(\R^D,\C^n)\to L^q(\R^D,\C^n)}=\|T\|_{L^p(\Om,\C^n)\to
L^q(\Om,\C^n)}$. In the following, we will shortly write
$\|\cdot\|_{p\to q}$ for the norm $\|\cdot\|_{L^p(\R^D,\C^n)\to
L^q(\R^D,\C^n)}$.

For the proof of part a), fix $t>0$ and $x,y\in\R^D$. In order to
verify (\ref{GGE-SG-Res-b}), we use weighted norm estimates for the
resolvent operators similar to those of Davies' perturbation method
presented in the previous section and an integral representation for
the semigroup operators based on the Cauchy formula.

Put $h\colon\R\to\R,\,h(\tau):=\beta\tau$ for some positive constant
$\beta$. Then one gets for the Legendre transform
$h^\#\colon\R\to[-h(0),\infty]$ of $h$
\begin{align}\label{Legendre}
h^\#(\si) := \sup_{\tau\geq0}\,\bigl(\si\tau-h(\tau)\bigr)
=\sup_{\tau\geq0}\,(\si-\beta)\tau = \left\{
\begin{array}{cl}
0&\mbox{for }\si\leq\beta\,,\\
\infty&\mbox{for }\si>\beta\,.
\end{array}
\right.
\end{align}

As before, $\E$ denotes the space of all real-valued functions
$\phi\in C_c^\infty(\R^D)$ with $\|\partial_j\phi\|_\infty\leq1$ for
any $j\in\{1,2,\ldots,D\}$. Then $d_\E(x,y) :=
\sup\{\phi(x)-\phi(y):\,\phi\in\E\}$ defines a metric on $\R^D$
which is actually equivalent to the Euclidean distance (cf., e.g.\
\cite[Lemma 4]{D95}). Therefore, \cite[Theorem 1.2]{BKLeg} is
applicable and gives that (\ref{GGE-SG-Res-a}) is equivalent to
$$
\bigl\|e^{-\vrh\phi}v^{\frac1p-\frac1q}_{|\la|^{-1/m}}
\la(\la+L)^{-1}e^{\vrh\phi}\bigr\|_{p\to q}
\leqC
e^{h^\#(\vrh|\la|^{-1/m})}\,,
$$
where $v_{|\la|^{-1/m}}(x):=|B(x,|\la|^{-1/m})|\cong|\la|^{-\frac
Dm}$, and consequently
$$
\bigl\|e^{-\vrh\phi}\la(\la+L)^{-1}e^{\vrh\phi}\bigr\|_{p\to q}
\leqC
|\la|^{\frac Dm(\frac1p-\frac1q)}e^{h^\#(\vrh|\la|^{-1/m})}
$$
for any $\la\in\C\setminus\{0\}$ with $|\arg\la|<\pi-\theta$,
$\vrh\geq0$, and any $\phi\in\E$. By exploiting (\ref{Legendre}), we
have for any $\la\in\C\setminus\{0\}$ with $|\arg\la|<\pi-\theta$,
$0\leq\vrh\leq\beta|\la|^{1/m}$, and $\phi\in\E$
\begin{align}\label{GGE-SG-Res-weight}
\bigl\|e^{-\vrh\phi}\la(\la+L)^{-1}e^{\vrh\phi}\bigr\|_{p\to q}
\leqC|\la|^{\frac Dm(\frac1p-\frac1q)}\,.
\end{align}
Based on the Cauchy integral formula, one can represent the
semigroup operator $e^{-tL}$ in terms of resolvent operators
\begin{align*}
e^{-tL}=\frac1{2\pi i}\int_{\Ga}e^{t\la}(\la+L)^{-1}\,d\la\,,
\end{align*}
where $\Ga$ is, as usual, a piecewise smooth curve in
$\Si_{\pi-\theta}$ going from $\infty e^{-i(\pi-\theta')}$ to
$\infty e^{i(\pi-\theta')}$ for some $\theta'\in(\theta,\pi/2)$.
Define $\eta:=\frac12(\pi-\theta+\frac\pi2)=\frac34\pi-\frac\theta2$
and $\om_\vrh:=|\sin\eta|^{-1}\,\beta^{-m}\vrh^m$ for $\vrh\geq0$
with $\beta$ being the constant in the definition of the function
$h$. We consider shifted versions of $e^{-tL}$ and shall establish a
bound on $\|e^{-\vrh\phi}e^{-\om_\vrh t}e^{-tL}e^{\vrh\phi}\|_{p\to
q} $ for any $\vrh\geq0$ and $\phi\in\E$ by using the above integral
representation for $e^{-tL}$ with the counterclockwise oriented
integration path $\Ga=\Ga_{t^{-1},\eta}+\om_\vrh$, where
$$
\Ga_{t^{-1},\eta}:=-(-\infty,-t^{-1}]e^{-i\eta}\,\cup\,t^{-1}
e^{i[-\eta,\eta]} \,\cup\,[t^{-1},\infty)e^{i\eta}\,.
$$
It holds for each $\vrh\geq0$ and $\phi\in\E$
\begin{align*}
&\bigl\|e^{-\vrh\phi}e^{-\om_\vrh t}e^{-tL}e^{\vrh\phi}
\bigr\|_{p\to q}
\leq
\int_{\Ga_{t^{-1},\eta}+\om_\vrh}e^{t(\Re\la-\om_\vrh)}
\bigl\|e^{-\vrh\phi}(\la+L)^{-1}e^{\vrh\phi}\bigr\|_{p\to q}
\,|d\la|
\\&\quad=
\int_{\Ga_{t^{-1},\eta}}\frac{e^{t\Re\zeta}}{|\zeta+\om_\vrh|}
\,\bigl\|e^{-\vrh\phi}(\zeta+\om_\vrh)(\zeta+\om_\vrh+L)^{-1}
e^{\vrh\phi}\bigr\|_{p\to q} \,|d\zeta|\,. \intertext{For every
$\zeta\in\Ga_{t^{-1},\eta}$ we can bound the operator norm with the
help of (\ref{GGE-SG-Res-weight}) when the condition
$\vrh\leq\beta\,|\zeta+\om_\vrh|^{1/m}$ is valid. A simple geometric
argument gives that $|\zeta+\om_\vrh|\geq|\sin\eta|\, \om_\vrh$ and
thus (\ref{GGE-SG-Res-weight}) surely applies for
$\vrh\leq\beta\,|\sin\eta|^{1/m}\,\om_\vrh^{1/m}$. But, due to the
definition of $\om_\vrh$, this requirement imposes no restrictions
on $\vrh$. Therefore, we can continue our estimation by applying
(\ref{GGE-SG-Res-weight}) and the elementary fact
$|\zeta+\om_\vrh|\cong|\zeta|+\om_\vrh$}
&\quad\leqC
\int_{\Ga_{t^{-1},\eta}}\frac{e^{t\Re\zeta}}{|\zeta|+\om_\vrh}\,
(|\zeta|+\om_\vrh)^{\frac Dm(\frac1p-\frac1q)}\,|d\zeta|
\leq
\int_{\Ga_{t^{-1},\eta}}{e^{t\Re\zeta}}\,
{|\zeta|}^{\frac Dm(\frac1p-\frac1q)-1}\,|d\zeta|\,.
\end{align*}
Here, we made use of the condition $D/m(1/p-1/q)<1$. Next, we
estimate the integral on each of the three segments of the
integration path $\Ga_{t^{-1},\eta}$ separately. We begin with a
bound for the integral on the half ray $[t^{-1},\infty)e^{i\eta}$
\begin{align*}
&\int_{[t^{-1},\infty)e^{i\eta}}e^{t\Re\zeta}\,
|\zeta|^{\frac Dm(\frac1p-\frac1q)-1}\,|d\zeta|
=
\int_{t^{-1}}^\infty e^{tu\cos\eta}\,u^{\frac Dm(\frac1p-\frac1q)-1}\,du
\\&\quad=
t^{-\frac Dm(\frac1p-\frac1q)}\int_{1}^\infty
e^{v\cos\eta}\,v^{\frac Dm(\frac1p-\frac1q)-1}\,dv
\leqC t^{-\frac Dm(\frac1p-\frac1q)}\,,
\end{align*}
where the last step is due to $\cos\eta<0$. The integral on the half
ray $-(-\infty,-t^{-1}]e^{-i\eta}$ can be treated in the same
manner. A bound for the remaining integral over the circular arc
$t^{-1}e^{i[-\eta,\eta]}$ is obtained by using the canonical
parametrization $\zeta(\al)=t^{-1}e^{i\al}$ for $\al\in[-\eta,\eta]$
\begin{align*}
\int_{t^{-1}e^{i[-\eta,\eta]}} e^{t\Re\zeta}\,
|\zeta|^{\frac Dm(\frac1p-\frac1q)-1} \,|d\zeta|
=
t^{-\frac Dm(\frac1p-\frac1q)}\int_{-\eta}^\eta e^{\cos\al}\,d\al
\leqC
t^{-\frac Dm(\frac1p-\frac1q)}\,.
\end{align*}
Putting things together, we have shown that for all $\vrh\geq0$,
$\phi\in\E$, and $t>0$
\begin{align*}
\bigl\|e^{-\vrh\phi}e^{-\om_\vrh t}e^{-tL}e^{\vrh\phi}\bigr\|_{p\to
q} \leqC t^{-\frac Dm(\frac1p-\frac1q)}
\end{align*}
and after recalling $\om_\vrh=|\sin\eta|^{-1}\,\beta^{-m}\vrh^m$
\begin{align*}
\bigl\|e^{-\vrh\phi}e^{-tL}e^{\vrh\phi}\bigr\|_{p\to q} \leqC
t^{-\frac Dm(\frac1p-\frac1q)}e^{|\sin\eta|^{-1}\,\beta^{-m}\vrh^m t}\,.
\end{align*}
By similar arguments as in the proof of \cite[Proposition 
8.22]{Levico}, this entails the desired two-ball estimate 
(\ref{GGE-SG-Res-b}).

\smallskip
The proof of part b) is similar to that of a) and is therefore
omitted.
\end{proof}

\section{The Lam\'e system}

Recently, M.\ Mitrea and S.\ Monniaux (\cite{MiMoLame}) studied 
properties of the Lam\'e system which appears in the linearization
of the compressible Navier-Stokes equations. They showed analyticity
of the semigroup generated by the Lam\'e operator and maximal
regularity for the time-dependent Lam\'e system equipped with
homogeneous Dirichlet boundary conditions. Their approach is
essentially based on off-diagonal estimates for the resolvents of
the Lam\'e operator. But according to Lemma \ref{GGE-SG-Res}, the
latter are basically equivalent to generalized Gaussian estimates
and this leads to further consequences for the Lam\'e system.

\smallskip
At first, we describe the setting of \cite{MiMoLame}. Although our
results apply in the general frame\-work of \cite{MiMoLame} as well,
we will restrict ourselves to the three-dimensional case. This
restriction serves only to introduce less notation. Furthermore, we
consider complex-valued functions. Let $\Om$ be a bounded, open
subset of $\R^3$ such that the {\em interior ball condition} holds,
i.e.\ there exists a positive constant $c$ such that for all
$x\in\Om$ and all $r\in(0,\frac12\,\mbox{diam}(\Om))$
$$
|B(x,r)|\geq cr^3\,.
$$
This condition ensures that $\Om$ becomes a space of homogeneous
type when $\Om$ is equipped with the three-dimensional Lebesgue
measure and the Euclidean distance. For example, any bounded
Lipschitz domain in $\R^3$ or domains satisfying an interior
corkscrew condition enjoy the interior ball condition.

Fix $\mu,\mu'\in\R$ with $\mu>0$ and $\mu+\mu'>0$. We consider
the sesquilinear form $\c$ defined by
$$
\c(u,v):=\mu\int_\Om\rot u\cdot\ov{\rot v}\,dx
+(\mu+\mu')\int_\Om\div u\,\,\ov{\div v}\,dx
$$
for $u,v\in H_0^1(\Om,\C^3)$, where $H_0^1(\Om,\C^3)$ denotes the
closure of the test function space $C_c^\infty(\Om,\C^3)$ with
respect to the norm of the Sobolev space $H^1(\Om,\C^3)$. Then it is
easy to see that the form $\c$ is closed, continuous, symmetric, and
coercive. Therefore, the operator $L$ associated with the form $\c$
is self-adjoint on $L^2(\Om,\C^3)$ and $-L$ generates a bounded
analytic semigroup on $L^2(\Om,\C^3)$. In \cite[Section
1.1]{MiMoLame} it is checked that $L$ is given by
\begin{align*}
\dom(L)&=\bigl\{u\in H_0^1(\Om,\C^3)\,:\,\mu\Delta u+\mu'\grad\div
u\in L^2(\Om,\C^3)\bigr\}\,,\\
Lu&=-\mu\Delta u-\mu'\grad\div u\qquad\mbox{for }u\in\dom(L)\,.
\end{align*}
The operator $L$ is called {\em Lam\'e operator with Dirichlet
boundary conditions}. In \cite[Section 2]{MiMoLame} M.\ Mitrea and
S.\ Monniaux adapt their approach of \cite{MiMo} to the Lam\'e
operator $L$ and establish the following statement: For any fixed
angle $\theta\in(0,\pi)$ there exist $q\in(2,\infty]$ and constants
$b,C>0$ such that for all $j\in\N$, $x\in\Om$, and
$\la\in\C\setminus\{0\}$ with $|\arg\la|<\pi-\theta$
\begin{align}\label{ODResLame}
\bigl\|\cf_{B(x,|\la|^{-1/2})}\la(\la+L)^{-1}
\cf_{B(x,2^{j+1}|\la|^{-1/2})\setminus B(x,2^{j-1}|\la|^{-1/2})}
\bigr\|_{2\to q}
\leq
C\,|\la|^{\frac 32(\frac12-\frac1q)}\,e^{-b\,2^j}\,.
\end{align}
Since this guarantees the validity of (\ref{GGE-SG-Res-a}), Lemma
\ref{GGE-SG-Res} yields generalized Gaussian $(2,q)$-estimates for
the Lam\'e operator $L$.

As remarked in \cite[Remark 1.5]{MiMoLame}, the estimate
(\ref{ODResLame}) is always valid for $q=6$ due to the
Sobolev embedding $H^1(\Om,\C^3)\hookrightarrow L^6(\Om,\C^3)$. If
the Poisson problem for the Lam\'e operator (cf.\
\cite[(1.15)]{MiMoLame}) is well-posed in $L^6(\Om,\C^3)$, then,
according to \cite[Lemma 2.2]{MiMoLame}, (\ref{ODResLame}) also
holds for $q^*=\infty $. It turns out that the largest value
$q_0\in(2,\infty]$, for which the iterative method of M.\ Mitrea and
S.\ Monniaux delivers (\ref{ODResLame}) and thus generalized
Gaussian $(2,q_0)$-estimates for $L$, depends on the well-posedness
of the Poisson problem for the Lam\'e operator and this is deeply
connected to the regularity properties of the boundary
$\partial\Om$. Only for certain domains $\Om$ the exact
characterization of $q_0$ is known. We refer to \cite[Theorem
4.1]{MiMoLame} for a discussion of this topic and only mention that,
if $\Om$ is a bounded Lipschitz domain in $\R^3$, then one can even
prove (\ref{ODResLame}) for $q=\infty$ (cf.\ \cite[Remark
1.6]{MiMoLame}), i.e.\ $L$ actually satisfies pointwise Gaussian estimates.
However, in general, (\ref{ODResLame}) with $q=\infty$ does not
hold. All in all, the Lam\'e operator $L$ fulfills generalized
Gaussian $(q_0',q_0)$-estimates for some $q_0\in[6,\infty]$.
Therefore, Theorem \ref{specmultthm} applies for $L$ and gives the
following result.
\begin{theorem}\label{SpecMultLame}
Fix $p\in(q_0',q_0)$. Suppose that $s>3|1/p-1/2|$. Then, for every
bounded Borel function $F\colon[0,\infty)\to\C$ with
$\sup_{n\in\Z}\| \om F(2^n\cdot)\|_{C^s}<\infty$, the operator
$F(L)$ is bounded on $L^p(\Om,\C^3)$ and there exists a constant
$C>0$ such that
$$
\|F(L)\|_{L^p(\Om,\C^3)\to L^p(\Om,\C^3)}\leq C\Bigl(\sup_{n\in\Z}\|\om
F(2^n\cdot)\|_{C^s}+|F(0)|\Bigr)\,.
$$
\end{theorem}

\end{document}